\documentclass{amsart}
\newtheorem{thm}{Theorem}[section]
\newtheorem{lem}[thm]{Lemma}

\usepackage{amssymb}

\usepackage{verbatim}

\numberwithin{equation}{section}

\begin{document}

\title{Generic Continuity of Metric Entropy for Volume-preserving Diffeomorphisms}

\author{Jiagang Yang}
\address{Departamento de Geometria, Instituto de Matem\'{a}tica e Estattica, Universidade Federal Fluminense, Niter\'{o}i, 24020-140, Brazil }

\email{yangjg@impa.br}
%\thanks{The first author was supported }

\author{Yunhua Zhou}
\address{College of Mathematics and Statistics, Chongqing University, Chongqing, 401331,
P. R. China}

\email{zhouyh@cqu.edu.cn}
\thanks{J.Y. is partially supported by CNPq, FAPERJ, and PRONEX. Y.Z. is the corresponding author and is partially supported  by Fundamental Research for Central Universities (CQDXWL2012008)}

%    General info

\date{\today}


\begin{abstract}
Let $M$ be a compact manifold and  $\text{Diff}^1_m(M)$ be the set of $C^1$ volume-preserving  diffeomorphisms of $M$.
We prove that there is a residual subset  $\mathcal {R}\subset \text{Diff}^1_m(M)$ such that each $f\in \mathcal{R}$  is a continuity point
of the map $g\to h_m(g)$ from  $\text{Diff}^1_m(M)$ to $\mathbb{R}$, where $h_m(g)$ is the metric entropy of $g$ with respect to volume measure $m$.
\end{abstract}

\maketitle

\section {Introduction}

Let $M$  be a smooth compact Riemannian
manifold with dimension $d$, and $m$ be a smooth volume
measure on $M$. Without loss of generality, we always assume that $m(M)=1$ in this paper.
 Denote by $\text{Diff}^r_m(M)$ the set of
$C^r$   volume-preserving
diffeomorphisms of $M$ endowed  with
$C^r$ topology for $r\geq 1$.

Our main result is

\begin{thm}\label{thm}
There is a residual subset  $\mathcal {R}\subset \text{Diff}^1_m(M)$ such that each $f\in \mathcal{R}$  is a continuity point
of the metric entropy map
$$
\begin{array}{rrcl}
\mathcal {E}:& \text{Diff}^1_m(M)&\to& \mathbb{R}\\
& g&\mapsto& h_m(g),
\end{array}
$$ where $h_m(g)$ is the metric entropy of $g$ with respect to volume measure $m$.
\end{thm}

The study of variation of entropy mainly focuses on two issues:
the continuity of topological entropies and of metric entropies. In generally, the variation of entropies is not even semicontinuous (e.g., see \cite{M}).
S. Newhouse(\cite{New}) proved that the metric entropy function
$$\mu\to h_u(f)$$
is upper semicontinuous for all $f\in \text{Diff}^\infty (M)$. In \cite{Yom} (see also \cite{B}), Y. Yomdin  proved that the topological entropy function
$$f\to h(f)$$
is upper semicontinuous on $\text{Diff}^\infty (M)$. Together with the result of A. Katok (\cite{K}), topological entropy is continuous for
$C^\infty$ systems on surface. Recently, G. Liao etc. (\cite{LVY}) extend the semicontinuity results of Newhouse and Yomdin to $C^1$ diffeomorphisms away from tangencies.

\section{Preliminaries}

\subsection{Lyapunov exponents and dominated splitting}
Given $f\in \text{Diff}^1_m(M)$, by Oseledec Theory, there is a
$m$-full invariant set $\mathcal{O}\subset M$ such that for every $x\in
\mathcal {O}$ there exist a splitting (which is called {\em Osledec
splitting})
$$T_xM= E_1(x)\oplus \cdots\oplus  E_{ k(x)}(x)$$
and real numbers (the {\em Lypunov exponents} at $x$) $\chi_1(x,f)>\chi_2(x, f)>\cdots>
\chi_{k(x)}(x,f)$ satisfying $Df( E_j(x))= E_j(fx)$ and
$$\lim\limits_{n\to \pm \infty}\frac{1}{n}\ln \|Df^nv\|=\chi_j(x,f)$$
for every $v\in  E_j(x)\setminus \{0\}$ and $j=1, 2, \cdots,  k(x)$.
 In the following,  by counting multiplicity, we also rewrite the Lyapunov
exponents of $m$ as $$\lambda_1(x,f)\geq \lambda_2(x,f)\geq \cdots \geq \lambda_d(x,f).$$

For $x\in \mathcal O$, we denote by
$$
\xi_i(x,f)=\left\{
\begin{array}{cl}
\lambda_i(x,f), &{\text {if }} \lambda_i(x,f)\geq 0;\\
0, &{\text {if }} \lambda_i(x,f)< 0
\end{array}
\right.
$$
and
$$\chi^+(x,f)=\sum \xi_i(x,f).$$

By the definitions, it is obviously that for $f,g\in\text{Diff}_m^1(M)$, one has
\begin{equation}\label{eq:10}
\int |\chi^+(x,f)-\chi^+(x,g)|dm(x)\leq \sum \int|\lambda_i(x,f)-\lambda_i(x,g)|dm(x)
\end{equation}

For $f\in \text{Diff}^1_m(M)$ and $\delta>0$, denote by $\mathcal{U}(f,\delta)$ the set of diffeomorphisms
$g\in \text{Diff}^1_m(M)$ such that the $C^1$ distance between $g$ and $f$ is less than $\delta$.

Given a diffeomorphism $f$, we say $Df$ has {\em a dominated splitting of index $i$ at a point} $x\in M$ if there are a $Df$-invariant splitting $T_{orb(x)}M=E\oplus F$ and a constant $N(x)\in \mathbb N$ such that $\dim(F)=i$ and
$$
\frac{\|Df^{N(x)}|_{E(f^j(x))}\|}{m(Df^{N(x)}|_{F(f^j(x))})}<\frac{1}{2}, \ \forall j\in \mathbb Z.
$$
We also denote the dominated splitting by $E\prec F$.

Let $\Lambda$ be an $f$-invariant set and $T_\Lambda M=E\oplus F$ be a $Df$-invariant splitting on $\Lambda$. We call $T_\Lambda M=E\oplus F$ be {\em a $N$-dominated splitting}, if there exists $N\in \mathbb N$ such that
$$
\frac{\|Df^{N}|_{E(y)}\|}{m(Df^{N }|_{E(y)})}<\frac{1}{2}, \ \forall y\in \Lambda.
$$

Let us note that the dominated splitting has {\em persistence property} (\cite{BDV}). That is, if $\Lambda$ is an $f$-invariant set with an $N$-dominated splitting, then there is a neighborhood $U$ of $\Lambda$ and a $C^1$-neighborhood $\mathcal U$ of $f$ such that for every $g\in \mathcal U$, the maximal $g$-invariant set in the closure of $U$ admits an $N$-dominated splitting, having the same dimensions of the initial dominated splitting over $K$.

\subsection{$C^1$ generic properties}
We recall three $C^1$ generic properties which will be used in the proof of Theorem \ref{thm}.

The first is about the relation of Osledec splitting and dominated splitting.

\begin{lem} \label{lem:BV} (Theorem 1 of \cite{BV})
There exists a residual set $R\subset \text{Diff}^1_m(M)$ such that, for each
$f\in R$ and a measurable function $N: M\rightarrow \mathbb{N}$ such that for $m$-almost every $x\in M$,
the Oseledets splitting of $f$ is either trivial or
is $N(x)$ dominated at $x$.

\end{lem}

The second one is the generic continuity of the Lyapunov spectrum.

\begin{lem} \label{lem:AB2} (Theorem D of \cite{AB})
Fix an integer $r\geq 1$. For each $i$, the continuous points of the map
$$
\begin{array}{rrcl}
\lambda_i:& \text{Diff}^r_m(M)&\to& L^1(M)\\
& f&\mapsto& \lambda_i(\cdot,f)
\end{array}$$
%$$\lambda_i: \text{Diff}^r_m(M)\to L^1(M)$$
form a residual subset.
\end{lem}

The third property is the generic persistence of invariant sets. It says that
if $f$ is a generic volume-preserving diffeomorphism, then its measurable invariant
sets persist in a certain (measure-theoretic and topological) sense
under perturbations of $f$.

\begin{lem} \label{lem:AB3} (Theorem C of \cite{AB})
Fix an integer $r\geq 0$. There is a residual set $\mathcal{R}\subset \text{Diff}^r_m(M)$ such that for every $f\in \mathcal {R}$, every $f$-invariant Borel set $\Lambda\subset M$ with positive volume, and every $\eta>0$, if $g\in  \text{Diff}^r_m(M)$ is sufficiently close to $f$ then there exists a $g$-invariant Borel set $\tilde{\Lambda}$ such that
$$\tilde{\Lambda}\subset B_\eta(\Lambda)\ \text{and }\ m(\tilde{\Lambda}\triangle \Lambda)<\eta,$$
here $B_\eta (\Lambda)=\{y\in M: \ d(x,y)<\eta \text{ for some } x\in \Lambda\}.$
\end{lem}

\subsection{$C^1$ Pesin entropy formula}
In \cite{ST}, W. Sun and X. Tian proved that the Pesin entropy formula holds for a generic $f\in \text{Diff}^1_m(M)$.

\begin{lem}\label{lem:ST}(Theorem 2.5 of \cite{ST})
 There exists a residual subset $\mathcal{R}\subset \text{Diff}^1_
m(M)$ such that for
every $f\in \mathcal{ R}$, the metric entropy $h_m(f)$ satisfies Pesin's entropy formula, i.e.,
$$h_m(f)=\int_M\chi^+(x,f)dm.$$
\end{lem}

In fact, Lemma \ref{lem:ST} is a corollary of Ruelle's inequality and the following result.

\begin{lem}\label{lem:ST2} (Theorem 2.2 of \cite{ST})
Let $f:M\to M$ be a $C^1$ diffeomorphism on a compact Riemannian manifold with dimension $d$. Let $f$ preserve an invariant probability $\mu$ which is absolutely continuous relative to Lebesgue measure. For $\mu$-a.e. $x\in M$, denote by
  $$\lambda_1(x)\geq \lambda_2(x)\geq \cdots\geq \lambda_d(x)$$
the Lyapunov exponents at $x$. Let $N(\cdot): M\to \mathbb{N}$ be an $f$-invariant measurable function. If for $\mu$-a.e. $x\in M$, there is a $N(x)$-dominated splitting: $T_{orb(x)}M=E\prec F$, then
$$h_\mu(f)\geq \int_M\chi_F(x)dm$$
here $\chi_F(x)=\sum_{i=1}^{\dim F(x)}\lambda_i(x)$.
\end{lem}
%The proof is essentially included in the proof of Theorem 2.2 of \cite{ST}.

\section{Proof of Theorem \ref{thm}}

\begin{proof}[Proof of Theorem \ref{thm}]
We will prove the Theorem by two steps. In step 1, we first prove that there is a residual subset  $\mathcal {R}_1\subset \text{Diff}^1_m(M)$ such that the entropy map $\mathcal E$
is  upper-semicontinuous at each $f\in \mathcal{R}_1$. In step 2, it will be proved that the set of lower-semicontinuous points of $\mathcal E$ contains a residual set $\mathcal {R}_2\subset \text{Diff}^1_m(M)$. Setting $\mathcal {R}=\mathcal{R}_1\cap \mathcal{R}_2$, we complete the proof of Theorem \ref{thm}.

{\bf Step 1.} Let $\mathcal{R}_1\subset \text{Diff}^1_m(M)$ satisfying Lemma \ref{lem:AB2} and Lemma \ref{lem:ST}. In this step, we will prove that, for any  $f\in \mathcal{R}_1$,
$$\limsup\limits_{g\to f}h_m(g)\leq h_m(f).$$

In fact,
by Lemma \ref{lem:ST},
$$h_m(f)=\int_M\chi^+(x,f)dm, \ \forall f\in \mathcal{R}_1.$$
So, by the well known  Ruelle's inequality
$$h_m(g)\leq \int_M\chi^+(x,g)dm, \ \forall g\in \text{Diff}^1_m(M)$$
 and (\ref{eq:10}), we have that for $\forall g\in \mathcal{U}(f,\delta)$,
\begin{equation}\label{eq:3}
h_m(g)-h_m(f)\leq \int_M | \chi^+(x,g)-\chi^+(x,f)|dm\\
\leq \sum \int_M |\lambda_i(x,f)-\lambda_i(x,g)| dm.
\end{equation}

Combining with Lemma~\ref{lem:AB2}, we proved the upper-semicontinuiation.

{\bf Step 2.} Let $\mathcal{R}_2\subset \mathcal{R}_1$ which satisfies Lemma~\ref{lem:BV} and Lemma~\ref{lem:AB3}. We will prove that  the entropy map $\mathcal E$
is  lower-semicontinus at each $f\in \mathcal{R}_2$. That is, for any $f\in \mathcal{R}_2$ and $\varepsilon>0$, there are positive numbers $\delta$ and $D$   such that
\begin{equation}\label{eq:16}
h_m(g)\geq h_m(f)-D\varepsilon, \ \ \forall g\in \mathcal{U}(f, \delta),
\end{equation}
here $D$ is only dependent on $d$ and $D_f$.

If the Oseledec splitting of $f$ is trivial on Lebesgue almost every point, then
 $h_m(f)=0$ and $h_m(g)\geq h_m(f)$ for all $g\in \text{Diff}^1_m(M)$. This means that the metric entropy map is lower semicontinuous at $f$.
So, in the following, we always assume that the Oseledec splitting of $f$ is not trivial.

Let
$$M_i(f)=\{x:\lambda_i(x,f)>0, \lambda_{i+1}(x,f)\leq 0\}.$$
Then
\begin{equation}\label{eq:11}
M_i(f)\cap M_j(f)=\emptyset, \ \forall \ i\neq j
\end{equation}
and
\begin{equation}\label{eq:14}
h_m(f)=\sum\limits_{i=1}^d \int_{M_i(f)}\sum\limits_{j=1}^i\lambda_j(x,f)dm.
\end{equation}
%Then $ \theta=\min\limits_{m(M_i)>0}\{m(M_i)\}\in (0,1]$.

{\bf Claim 1.} For any $\varepsilon>0$, there is $\delta_1>0 $ such that for any $g\in\mathcal{U}(f,\delta_1)$ and $i=1,2,\cdots,d$, there exists $M'_i(f)\subset M_i(f)$ such that $m(M_i(f)\setminus M'_i(f))<\varepsilon$ and
$$\lambda_i(x,g)>0,\ \ \ m\text{-}a.e.\ x\in M'_i(f).$$

\begin{proof}[Proof of Claim 1.]
For any $\varepsilon>0$ and $i=1,2,\cdots, n$, there is $k(i)>0$ such that
\begin{equation}\label{eq:8}
m(M_i(f)\setminus M_{ik(i)}(f))<\frac{\varepsilon}{2}
\end{equation}
here
$$M_{ik(i)}(f)=\{x\in M_i(f): \lambda_i(x,f)\geq \frac{1}{k(i)}\}.$$

Let $\varepsilon'=\min\{\varepsilon,\frac{\varepsilon}{2k(1) }, \cdots, \frac{\varepsilon}{2k(n) }\}.$
By Lemma \ref{lem:AB2}, there is $\delta_1>0$ such that for any $g\in\mathcal{U}(f,\delta_1)$ and any $i$,
$$\int_M|\lambda_i(x,g)-\lambda_i(x,f)|dm<\varepsilon'$$

Set $$M'_{i }(f)=\{x\in M_{ik(i)}(f): \lambda_i(x,g)>0\}\ \text{ and }\ M''_{i }(f)=M_{ik(i)}(f)\setminus M'_{i }(f).$$
Then we have
\begin{equation}\label{eq:9}
m(M''_{i }(f))\leq  \frac{\varepsilon}{2}.
\end{equation}
%Suppose that $\lambda_i(x,g)\leq 0$ on $M_{ik(i)}(f)$. Then
In fact, if $m(M''_{i}(f))> \frac{\varepsilon}{2}$, we have
$$\int_{M}|\lambda_i(x,f)-\lambda_i(x,g)|dm\geq\int_{ M''_{i}(f)}|\lambda_i(x,f)-\lambda_i(x,g)|dm\geq \frac{1}{k(i)}m(M''_{i}(f))>\varepsilon'.$$
This is a contradiction.
\end{proof}

{\bf Claim 2.} For any $\varepsilon>0$, there is $\delta_2>0 $ such that for any $g\in\mathcal{U}(f,\delta_2)$ and any $i=1,2,\cdots,d$, there exist $N\in \mathbb{N}$ and a $g$-invariant set $\tilde M_i(g)\subset M$ such that

(1) there is a $N$-dominated splitting of index $i$;

(2) $m(\tilde M_i(g) \triangle M_i(f))<\frac{\varepsilon}{2}$.

\begin{proof}[Proof of Claim 2.]
By Lemma~\ref{lem:BV}, there is a $N(x)$-dominated splitting of index $i$ at each $x\in M_{i }(f)$, for any $\varepsilon>0$, there are $N\in \mathbb{N}$ and an $f$-invariant subset $\tilde{M}_{i }(f)\subset M_{i }(f)$ such that
$m(M_{i }(f)\setminus \tilde{M}_{i }(f))<\frac{\varepsilon}{4}$ and
there is a $N$-dominated splitting of index $i$ at each $x\in \tilde{M}_{i}(f)$.

By the persistence property of dominated splitting and Lemma \ref{lem:AB3}, for any $\varepsilon>0$, there is $\delta_2>0$ such that for any $g\in \mathcal{U}(f,\delta_2)$ there is $g$-invariant set $\tilde{M}_i(g)$ closing to $\tilde{M}_{i }(f)$ such that there is a $N$-dominated splitting of index $i$ at each $x\in \tilde{M}_{i}(g)$ for $Dg$ and $m(\tilde{M}_{i}(g) \triangle \tilde{M}_{i}(f))< \frac{\varepsilon}{4}$.
\end{proof}

Set
$$M^+_i(g)=\{x\in \tilde M_i(g): \ \lambda_i(x,g)>0\}\ \text{ and }\ M^+(g)=\bigcup\limits_{i=1}^nM^+_i(g).$$
Then $M^+_i(g)$ and $M^+(g)$ are  $g$-invariant. By (2) of Claim 2 and (\ref{eq:11}), for any $g\in \mathcal{U}(f,\delta_2)$ and any $i\neq j$, we have $$m(M^+_i(g)\cap M^+_j(g))<\varepsilon$$
and so
\begin{equation}\label{eq:13}
m(M^+_i(g)\setminus \bigcup\limits_{j=1}^{i-1}M^+_j(g))\geq m(M_i^+(g))-(i-1)\varepsilon.\end{equation}
Furthermore, noting
$$\tilde M_i(g)\cap M'_i(f)\subset M^+_i(g)\subset \tilde M_i(g),$$
by Claim 1 and Claim 2, we have
$$m(M^+_i(g)\triangle M_i(f))<3\varepsilon, \ \forall g\in \mathcal{U}(f,\delta),$$ here $\delta=\min\{\delta_1, \delta_2\}$. Since $m(\cup_{i=1}^d M_i(f))=1$, it holds that
$$m( M^+(g))\geq 1-3d\varepsilon.$$

Now, we turn to estimate $h_m(g)$ for $g$ $C^1$-close to $f$.

For $g\in \mathcal{U}(f,\delta)$, we have
$$\int_{M^+_i(g)}|\lambda_j(x,g)-\lambda_j(x,f)|dm\leq \int_{M}|\lambda_j(x,g)-\lambda_j(x,f)|dm<\varepsilon,\ \forall j=1,2,\cdots,n.$$
So,
\begin{equation}\label{eq:15}
\begin{array}{cl}
&\int_{M^+_i(g)}\sum\limits_{j=1}^i\lambda_j(x,g)dm\\
>&\int_{M^+_i(g)}\sum\limits_{j=1}^i\lambda_j(x,f)dm
-i\varepsilon\\
=&\left(\int_{M^+_i(g)\setminus M_i(f)}
+\int_{M^+_i(g)\cap M_i(f)}\right)\sum\limits_{j=1}^i\lambda_j(x,f)dm
-i\varepsilon\\
\geq & \int_{M^+_i(g)\cap M_i(f)} \sum\limits_{j=1}^i\lambda_j(x,f)dm-iD_fm(M^+_i(g)\setminus M_i(f)) -i\varepsilon\\
=&\left(\int_{M_i(f)}
-\int_{M_i(f)\setminus M^+_i(g)}\right)\sum\limits_{j=1}^i\lambda_j(x,f)dm-iD_fm(M^+_i(g)\setminus M_i(f))
-i\varepsilon\\
\geq & \int_{M_i(f)}\sum\limits_{j=1}^i\lambda_j(x,f)dm-iD_f(m(M^+_i(g)\setminus M_i(f))+m(M_i(f)\setminus M^+_i(g)))
-i\varepsilon\\
\geq & \int_{M_i(f)}\sum\limits_{j=1}^i\lambda_j(x,f)dm-(3iD_f+i)\varepsilon
\end{array}
\end{equation}

Then by (\ref{eq:13}), (\ref{eq:15}) and Lemma \ref{lem:ST2},
$$
\begin{array}{ccl}
h_m(g)&\geq&\sum\limits_{i=1}^d \int_{M^+_i(g)\setminus \cup_{l=1}^{i-1}M^+_l(g)}\sum\limits_{j=1}^i\lambda_j(x,g)dm\\
&\geq&\sum\limits_{i=1}^d \left(\int_{M^+_i(g)}\sum\limits_{j=1}^i\lambda_j(x,g)dm-(i-1)\varepsilon D_f\right)\\
&\geq&\sum\limits_{i=1}^d \left(\int_{M_i(f)}\sum\limits_{j=1}^i\lambda_j(x,f)dm -(4D_f+1)i\varepsilon\right)\\
&=&h_m(f)-\frac{d(1+d)(4D_f+1)}{2}\varepsilon.
\end{array}
$$
Setting $D=3d^2D_f+\frac{d(1+d)(4D_f+1)}{2}$, we completes the proof of (\ref{eq:16}).

\end{proof}

\bibliographystyle{amsplain}

\end{document}